\documentclass[reqno]{amsart}

\usepackage{amsmath}
\usepackage{amsfonts}
\usepackage{amsthm}
\usepackage{amssymb}
\usepackage{graphicx}
\usepackage[rightcaption]{sidecap}
\usepackage{bm}
\usepackage{dsfont}
\usepackage{color}
\usepackage{hyperref}
\usepackage{enumerate}
\usepackage{comment}
\usepackage[font=footnotesize]{caption}
\usepackage[all]{xy}
\usepackage{xy}
\usepackage{tikz}
\usepackage{paralist}
\allowdisplaybreaks
\usepackage{pgfplots}
\pgfplotsset{%
   every tick label/.append style = {font=\tiny},
   every axis label/.append style = {font=\scriptsize}
}

\usepackage{a4wide}

\numberwithin{equation}{section}

\newtheorem{thm}{Theorem}[section]
\newtheorem*{thm*}{Theorem}
\newtheorem*{conj*}{Conjecture}
\newtheorem{cor}[thm]{Corollary}

\theoremstyle{definition}

\newtheorem{qst}[thm]{Question}

\newtheorem{dfn}[thm]{Definition}
\newtheorem{rmk}[thm]{Remark}

\newcommand{\di}{\mathrm{d}}
\newcommand{\R}{\mathbb{R}}
\newcommand{\Z}{\mathbb{Z}}

\begin{document}

\title[An observation about conformal points on surfaces]{An observation about conformal points on surfaces}

\author[P. Albers]{Peter Albers}
\address{Universit\"at Heidelberg, Institut f\"ur Mathematik, \newline\indent Im Neuenheimer Feld 205, 69120 Heidelberg, Germany}
\email{palbers@mathi.uni-heidelberg.de}
\author[G. Benedetti]{Gabriele Benedetti}
\address{VU Amsterdam, Faculty of Sciences, Department of Mathematics,\newline\indent De Boelelaan 1111, 1081 HV Amsterdam, The Netherlands}
\email{g.benedetti@vu.nl}
\date{\today}

\begin{abstract} We study the existence of points on a compact oriented surface at which a symmetric bilinear two-tensor field is conformal to a Riemannian metric. We give applications to the existence of conformal points of surface diffeomorphisms and vector fields.
\end{abstract}

\maketitle
\section{Statement of results}\label{s:one}
\subsection{Conformal points}
Let $\Sigma$ be a compact, oriented surface, possibly with non-empty boundary $\partial \Sigma$. Denote by $C_1,\ldots,C_n$ the boundary components of $\Sigma$ with the induced orientation. Let $\mathrm{Sym}((T^*\Sigma)^{\otimes2})\to\Sigma$ be the bundle of symmetric bilinear tensors on $\Sigma$. Fix a Riemannian metric $g$ on $\Sigma$, that is, a positive-definite section of $\mathrm{Sym}((T^*\Sigma)^{\otimes2})\to\Sigma$.
\begin{dfn}
We say that a section $h$ of $\mathrm{Sym}((T^*\Sigma)^{\otimes2})\to\Sigma$ is \textit{conformal to $g$ at the point $z\in \Sigma$} if there exists $c\in\R$ such that $h_z=cg_z$.
\end{dfn}
Motivated by \cite{AT}, the goal of this note is to study the set of points 
\[
\mathcal C(g,h)\subset\Sigma
\]
at which $h$ is conformal to $g$ (see Theorem \ref{t:main1}) in order to investigate conformal points of diffeomorphisms $F\colon\Sigma\to\Sigma$, in which case $h=F^*g$ (see Theorem \ref{t:main2} and Corollary \ref{c:main}), and of vector fields (see Corollary \ref{c:vf}).

Our main observation is that $\mathcal C(g,h)$ is the zero-set of a section $H^a$ in a distinguished vector bundle $E^a\to\Sigma$ over the surface, which we describe now. Let $\mathrm{End}(T\Sigma)\to\Sigma$ be the bundle of endomorphisms of $T\Sigma$ and let
\begin{equation}\label{eqn:bundle_E_A}
E^a\subset \mathrm{End}(T\Sigma)
\end{equation}
be the subbundle of those endomorphisms which are symmetric with respect to $g$ and have zero trace. For all $z\in\Sigma$, an element of $R\in E^a_z$ has the matrix expression
\[
\begin{pmatrix}
a&b\\ b&-a
\end{pmatrix}\qquad a,b\in\R,
\]
with respect to a positive, orthonormal basis of $T_z\Sigma$. Thus, any non-zero element $R\in E^a_z$ is, up to a positive scalar multiple, a reflection $R\colon T_z\Sigma\to T_z\Sigma$ along a line in $T_z\Sigma$. In particular, the $S^1$-bundle associated with $E^a$ is the bundle of unoriented lines in $T\Sigma$. This $S^1$-bundle is doubly covered by the bundle of oriented lines in $T\Sigma$ which, in turn, is the unit-tangent bundle  of $\Sigma$, that is, the $S^1$-bundle associated with $T\Sigma\to\Sigma$. The above discussion shows that $E^a$ is an oriented plane bundle over $\Sigma$ with Euler number
\begin{equation}\label{e:chi}
	e(E^a)=2e(T\Sigma)=2\chi(\Sigma).
\end{equation}
Given a symmetric bilinear two-tensor field $h$ over $\Sigma$, let $H$ be the section of $\mathrm{End}(T\Sigma)$ representing $h$ with respect to $g$, namely
\begin{equation}\label{e:H}
	g_z(u,H_zv)=h_z(u,v),\qquad \forall z\in\Sigma,\ \forall\,u,v\in T_z\Sigma.
\end{equation}
We denote by
\begin{equation}\label{e:Ha}
H^a:=H-\frac{\mathrm{tr} H}{2}I
\end{equation}
the section of $E^a$ corresponding to the trace-free part of $H$. Here $I$ is the section of $\mathrm{End}(T\Sigma)$ such that $I_z$ is the identity of $T_z\Sigma$ for all $z\in\Sigma$.

Thus, we conclude that
\[
z\in\mathcal C(g,h)\quad\iff\quad H^a_z=0.
\]
From this relationship we see that, generically, $h$ has only finitely many conformal points and all of them lie in the interior of $\Sigma$. In this case, we can use the Poincar\'e--Hopf Theorem for unoriented line fields on oriented surfaces with boundary to algebraically count conformal points. To give the precise statement, let us introduce some notation under the assumption that $\mathcal C(g,h)$ is finite and $\mathcal C(g,h)\subset\Sigma\setminus\partial\Sigma$. For each $z\in\mathcal C(g,h)$, we define 
\[
\mathrm{ind}_{(g,h)}(z)\in\Z
\]
as the index of $z$ seen as a zero of the section $H^a$ of $E^a\to\Sigma$. We count the elements in $\mathcal C(g,h)$ algebraically via the integer
\begin{equation}\label{e:Z}
	[\mathcal C(g,h)]:=\sum_{z\in \mathcal C(g,h)}\mathrm{ind}_{(g,h)}(z)\in\Z.
\end{equation}
Moreover, for every boundary component $C_i$ of $\Sigma$, with $i=1,\ldots,n$, we define 
\begin{equation}\label{e:w}
w_i(g,h)\in\Z
\end{equation}
as the winding number of the section $H^a|_{C_i}$ with respect to $R^i\in E^a$, where $R^i(z)$ is the reflection along the line $T_z\partial\Sigma\subset T_z\Sigma$ for $z\in C_i$.
\begin{thm}\label{t:main1}
Let $g$ be a Riemannian metric on a compact, oriented surface $\Sigma$. Then the following two statements hold.
\begin{enumerate}
\item For any symmetric bilinear two-tensor field $h$ over $\Sigma$ such that $\mathcal C(g,h)$ is finite and $\mathcal C(g,h)\subset \Sigma\setminus\partial\Sigma$, the equality
\begin{equation}\label{e:main1}
	[\mathcal C(g,h)]=2\chi(\Sigma)+\sum_{i=1}^nw_i(g,h)
\end{equation}
holds, where $\chi(\Sigma)$ denotes the Euler characteristic of $\Sigma$.
\item Let $\mathcal C\subset \Sigma\setminus\partial\Sigma$ be a finite set of points, $\iota\colon \mathcal C\to\mathbb Z$ an arbitrary function, and $w_1,\dots,w_n\in \mathbb Z$ arbitrary integers satisfying
\begin{equation}\label{e:data}
	\sum_{z\in \mathcal C}\iota(z)=2\chi(\Sigma)+\sum_{i=1}^nw_i.
\end{equation}
Then there exists a symmetric bilinear two-tensor field $h$ over $\Sigma$ such that $\mathcal C=\mathcal C(g,h)$, $\iota(z)=\mathrm{ind}_{(g,h)}(z)$ for all $z\in\mathcal C$ and $w_i=w_i(g,h)$ for all $i=1,\ldots,n$.
\end{enumerate}
\end{thm}
\begin{rmk}
For the convenience of the reader, we give the short proof of Theorem \ref{t:main1} in Section \ref{s:two} although this can be deduced from the literature. For statement (1), we refer to \cite{Morse}, \cite{Pugh} and \cite{Gottlieb} which deal with the Poincar\'e--Hopf Theorem for oriented line fields on surfaces with boundary and to \cite[III.2.2]{Hopf}, \cite{Markus}, \cite{Jaenich} and \cite{CG} which deal with the Poincar\'e--Hopf theorem for unoriented line fields on surfaces without boundary. For statement (2), we refer to the Extension Theorem in \cite[p.~145]{GP}. Finally, we notice that, passing to the orientation double cover, Theorem \ref{t:main1} also holds for non-orientable surfaces.
\end{rmk}
We discuss now two situations where the set $\mathcal C(g,h)$ naturally appears. 
\subsection{Carath\'eodory's conjecture}
First, let us consider a smooth embedding $\rho\colon S^2\to\R^3$. Here $\Sigma=S^2$ and we take $g^\rho$ and $h^\rho$ to be the first and the second fundamental form of the embedding $\rho$, respectively, with respect to the ambient Euclidean metric. The elements of $\mathcal C(g^\rho,h^\rho)$ are the so-called umbilical points, namely points at which the two principal curvatures of the embedding coincide. In this case, \eqref{e:main1} yields the well-known result that $[\mathcal C(g^\rho,h^\rho)]=4$, namely that the algebraic count of umbilical points is equal to four.
For example, when $\rho$ is an ellipsoid of revolution, $\mathcal C(g^\rho,h^\rho)$ consists exactly of the two poles, both having index two. In general, it is natural to ask which further conditions must the points $z\in \mathcal C(g^\rho, h^\rho)$ and their indices satisfy besides $[\mathcal C(g^\rho,h^\rho)]=4$. For instance, Carath\'eodory's conjecture \cite{GH,SG} asserts that convexity of the embedding $\rho$ implies $\mathrm{ind}(z)\leq 2$ for all $z\in\mathcal C(g^\rho,h^\rho)$, and, in particular, entails that $\mathcal C(g^\rho,h^\rho)$ always contains at least two points.

\subsection{Conformal points of a diffeomorphism}
The second situation in which $\mathcal C(g,h)$ naturally appears is when $h=F^*g$, where $F\colon\Sigma\to\Sigma$ is any orientation-preserving diffeomorphism of $\Sigma$. In this case, $\mathcal C(g,F^*g)$ is the set of so-called conformal points of $F$ (with respect to $g$). Assuming that $\mathcal C(g,F^*g)$ is finite and $\mathcal C(g,F^*g)\subset\Sigma\setminus\partial\Sigma$, we are going to give a formula for $w_i(g,F^*g)$ in terms of the behavior of $F$ at the boundary. In order to state the result, for $i=1,\ldots,n$ let $\nu_i\colon C_i\to T\Sigma$ be the outward normal at the boundary component $C_i$ and  $\tau_i\colon C_i\to T\Sigma$ be the unit vector tangent to $C_i$ in the positive direction. The pair $(\nu_i,\tau_i)$ then forms a positive orthonormal frame for $g$ along $C_i$. We trivialize $T\Sigma|_{\partial\Sigma}=\sqcup_{i}C_i\times \R^2$ using $(\nu_i,\tau_i)$ at $C_i$, $i=1,\ldots,n$. Since $F$ maps boundary components to boundary components (not necessarily the same) we can express $\di F$ in this trivialization as 
\begin{equation}\label{e:dF}
\di F\big|_{C_i}=:N_i=c_i\begin{pmatrix}
	a_i&0\\b_i&1
\end{pmatrix}.
\end{equation}
Here $a_i,c_i\colon C_i\to(0,\infty)$, $b_i\colon C_i\to\R$ and $(a_i,b_i)$ is never equal to $(1,0)$ since $F$ has no conformal point on $C_i$ by assumption. 
\begin{thm}\label{t:main2}
For all $i=1,\ldots,n$ we have the equality
\begin{equation}\label{e:main2}
w_i(g,F^*g)=w(a_i-1,b_i),
\end{equation}
where $w(a_i-1,b_i)$ is the winding number of the curve $(a_i-1,b_i)\colon C_i\cong S^1\to\R^2\setminus\{0\}$ around the origin.
\end{thm}
This formula, which will be proved in Section \ref{s:three}, allows us to compute $w_i(g,F^*g)$ if we understand the behavior of $F$ at points on the boundary sufficiently well. A remarkable example of this phenomenon is illustrated by the next corollary.
\begin{cor}\label{c:main}
If $F\colon\Sigma\to \Sigma$ is the identity on the boundary and preserves an area form on $\Sigma$, then \begin{equation}\label{e:cor1}
w_i(g,F^*g)=0,\qquad \forall\, i=1,\ldots,n.
\end{equation}
It follows that for this type of diffeomorphisms
\begin{equation}\label{e:cor2}
[\mathcal C(F)]=2\chi(\Sigma),
\end{equation}
that is, the number of conformal points of such an $F$ is twice the Euler characteristic.
\end{cor}
\begin{proof}
By \eqref{e:main2} the assertion is equivalent to showing $w_i(a_i-1,b_i)=0$. Since $F$ is the identity at the boundary we conclude that $\mathrm dF\cdot \tau_i=\tau_i$ and thus $c_i=1$ in \eqref{e:dF}. Since $F$ preserves an area form, it follows that $\det N_i=1$, which implies that $a_i=1$ in \eqref{e:dF}. Therefore, the curve $(a_i-1,b_i)=(0,b_i)$ is contained in the $y$-axis and does not cross $0$. We conclude that its winding number around the origin $w(a_i-1,b_i)$ vanishes.
\end{proof}
\begin{rmk}
Equation \eqref{e:cor2} was proved in \cite{AT}, when $\Sigma=D^2$, and $F$ satisfies some additional conditions, which hold, for instance, when $F$ is $C^1$-close to the identity, 
\end{rmk}
If we linearize the property of being a conformal point for a diffeomorphism at the identity of $\Sigma$, we get a corresponding condition for conformal points of vector fields on $\Sigma$. This condition is easier phrased after reinterpreting conformality in terms of complex geometry, as we explain next.
\subsection{Conformal points and complex structures}
Let $\jmath$ be the complex structure associated with the Riemannian metric $g$ and the orientation of $\Sigma$. In other words, $\jmath$ yields a section of $\mathrm{End}(T\Sigma)$ such that $v$ and $\jmath_zv$ form a positive, orthogonal basis of $T_z\Sigma$ for all $z\in\Sigma$ and all $v\in T_z\Sigma\setminus\{0\}$. Thus, $\jmath_z$ has the matrix expression
\begin{equation}\label{e:j}
\begin{pmatrix}
	0&-1\\ 1&0
\end{pmatrix}
\end{equation}
with respect to a positive, orthonormal basis of $T_z\Sigma$. An endomorphism $H\colon T_z\Sigma\to T_z\Sigma$ commutes with $\jmath_z$ if and only if $H$ has the matrix expression
\[
\begin{pmatrix}
	a&-b\\ b&a
\end{pmatrix}\qquad a,b\in\R,
\]
in such a basis. In particular, we deduce that $H$ is, up to a scalar multiple, a rotation matrix. Analogously, $H$ anticommutes with $\jmath_z$ if and only if $H$ has the matrix expression
\[
\begin{pmatrix}
	a&b\\ b&-a
\end{pmatrix}\qquad a,b\in\R,
\]
in such a basis. In particular, we deduce that $E^a$, see \eqref{eqn:bundle_E_A}, is exactly the bundle of endomorphisms anticommuting with $\jmath$. Therefore, if we denote by $E^c\to\Sigma$ the bundle of endomorphisms commuting with $\jmath$, we get the splitting
\begin{equation}\label{e:splitting_End_TSigma}
	\mathrm{End}(T\Sigma)=E^c\oplus E^a.
\end{equation}
Furthermore, as $\jmath$-complex line bundle we can write
	\[
	E^a\cong T\otimes \overline{T^*}, \qquad T:=T^{(1,0)}\Sigma,
	\]
	where $T^{(1,0)}\Sigma$ is the holomorphic tangent bundle of $\Sigma$ and $\overline{T^*}$ denotes the conjugate of the dual bundle of $T$. With this identification, a local section of $E^a$ is given by $\frac{\partial}{\partial z}\otimes\di\bar z$ where $z$ is a local holomorphic coordinate compatible with $\jmath$. Thus, the Euler number of $E^a$ as real oriented plane bundle coincides with its Chern number as complex line bundle. Using that $c_1(T)=\chi(\Sigma)$, this gives another derivation of \eqref{e:chi} by computing
	\[
	c_1(E^a)=c_1(T\otimes \overline{T^*})=c_1(T)-c_1(T^*)=c_1(T)+c_1(T)=2c_1(T)=2\chi(\Sigma).
	\]
Finally, let us assume that $z$ is a conformal point of an orientation-preserving  diffeomorphism $F\colon\Sigma\to\Sigma$. Then, 
\begin{equation}\label{e:cc}
(F^*g)_z=cg_z\ \text{for some }c>0. 
\end{equation}
If we denote by $M$ the matrix representation of $\mathrm d_zF$ with respect to positive, orthonormal bases of $T_z\Sigma$ and $T_{F(z)}\Sigma$, then \eqref{e:cc} can be rewritten as 
\[
M^TM=cI\,.
\]
This condition is equivalent to saying that $M$ is, up to a scalar multiple, a rotation matrix. Since $\jmath_z$ and $\jmath_{F(z)}$ are represented by the matrix \eqref{e:j}, we see that $\mathrm d_zF\jmath_z=\jmath_{F(z)}\mathrm d_zF$. We conclude that $z$ is a conformal point of $F$ if and only if $F$ is $\jmath$-holomorphic at $z$ with respect to $\jmath$-holomorphic coordinates around $z$ and $F(z)$.

\subsection{Conformal points of vector fields}
Let $f$ be a vector field on $\Sigma$. Let $F_t:\Sigma\to\Sigma$ be the time-$t$ map of the flow of $f$. Suppose that $z\in\Sigma$ is a point such that $F_t(z)\in\mathcal C(g,(F_t)^*g)$ for all $t$ close to zero. In particular, $F_t$ is $\jmath$-holomorphic at $z$ in a local $\jmath$-holomorphic chart for all small $t$. Taking the derivative in $t$ at $t=0$, we conclude that the vector field $f$ is $\jmath$-holomorphic at $z$. In other words $\bar\partial_{\jmath} f$ is a section of $E^a$ which vanishes at $z$. Here, $\bar\partial_{\jmath}$ denotes the Cauchy--Riemann operator sending sections of the holomorphic tangent bundle $T=T^{(1,0)}\Sigma$ to sections of $T\otimes \overline{T^*}\cong E^a$, and $f$ is identified with its image under the isomorphism
\[
T\Sigma\to T^{\mathbb C}\Sigma\cong T^{(1,0)}\Sigma\oplus T^{(0,1)}\Sigma\to T^{(1,0)}\Sigma,
\]
where $T^{\mathbb C}\Sigma$ is the complexification of $T\Sigma$. 

Let $\mathcal C(\jmath,f)$ be the set of zeros of $\bar\partial_{\jmath}f$. If $\mathcal C(\jmath,f)$ is finite and $\mathcal C(\jmath,f)\subset\Sigma\setminus\partial \Sigma$, then we can associate an index $\mathrm{ind}_{(\jmath,f)}(z)$ to each $z\in \mathcal C(\jmath,f)$ and a winding number $w_i(\jmath,f)$ representing the relative winding number of $\bar{\partial}_\jmath f$ with respect to the canonical section $R_i$ along $C_i$ for every $i=1,\ldots,n$. Defining the algebraic count 
\[
	[\mathcal C(\jmath,f)]:=\sum_{z\in \mathcal C(\jmath,f)}\mathrm{ind}_{(\jmath,f)}(z),
\]
we get the following consequence of Theorem \ref{t:main1}.(1).
\begin{cor}\label{c:vf}
Let $\jmath$ be a complex structure on a compact surface $\Sigma$ and $f$ a vector field on $\Sigma$ such that $\mathcal C(\jmath,f)$ is finite and $\mathcal C(\jmath,f)\subset\Sigma\setminus\partial \Sigma$. Then the equation
\begin{equation*}
	[\mathcal C(\jmath,f)]=2\chi(\Sigma)+\sum_{i=1}^nw_i(\jmath,f)
\end{equation*}
holds.
\end{cor}
\subsection{An open question}
Given any Riemannian metric $g$ on $\Sigma$ and diffeomorphism $F\colon\Sigma\to\Sigma$, it is interesting to ask which further restrictions must the points $z$ of $\mathcal C(g,F^*g)$, their indices $\mathrm{ind}_{(g,F^*g)}(z)$ and the numbers $w_i(g,F^*g)$ satisfy besides equation \eqref{e:main1}. This question is related to the uniformization theorem for compact surfaces with boundary via Theorem \ref{t:main1}.(2). For instance, given any two metrics $g$ and $h$ on $\Sigma=S^2$ or $\Sigma=D^2$, we can find a diffeomorphism $F\colon\Sigma\to\Sigma$ such that $F^*g$ and $h$ are conformal at every point \cite[Theorem 1]{OPS}. Thus, $\mathcal C(g,F^*g)=\mathcal C(g,h)$, $\mathrm{ind}_{(g,F^*g)}(z)=\mathrm{ind}_{(g,h)}(z)$ for every $z$ in this set, and $w_i(g,F^*g)=w_i(g,h)$ for all $i=1,\ldots,n$. As a consequence of Theorem \ref{t:main1}.(2), there are no further restrictions in this case.

On the other hand, on a general surface $\Sigma$ there are metrics $g$ and $h$ such that $h$ and $F^*g$ are not conformal at all points, no matter how we choose the diffeomorphism $F$. The easiest examples where this happens is when $\Sigma=\mathbb T^2$, or when $\Sigma=D^2$ and we require in addition the diffeomorphism $F$ to be the identity at the boundary. For instance, on $\mathbb T^2$ conformal classes of metrics $g$ are classified by lattices $\Gamma$ in $\mathbb C$, up to Euclidean isometries and homotheties, where $g$ is the Riemannian metric on $\mathbb T^2=\mathbb C/\Gamma$ induced by the Euclidean metric on $\mathbb C$. To get an example on the disc, let us identify $D^2$ with the unit Euclidean disc in $\mathbb C$. Let $g$ be the Euclidean metric on $D^2$. Recall that the group of diffeomorphisms $\varphi\colon D^2\to D^2$ such that $g$ and $\varphi^*g$ are conformal at all points consists of the Möbius transformations preserving $D^2$. Consider $G\colon D^2\to D^2$ to be any diffeomorphism such that $G|_{\partial D^2}\neq \varphi|_{\partial D^2}$ for all $\varphi$. Such a $G$ surely exists since if $\varphi$ is not the identity, then $\varphi$ can have at most two fixed points on the boundary. If we define $h:=G^*g$, then there is no diffeomorphism $F\colon D^2\to D^2$ which is identity at the boundary and such that $F^*h$ and $g$ are conformal at every point. Indeed, if such an $F$ exists, then $(G\circ F)^*g=F^*G^*g=F^*h$ is conformal to $g$ at all points, which means that $F\circ G=\varphi$ for some Möbius transformation $\varphi$ preserving the disc. Since $F$ is the identity at the boundary, this would imply that $G=\varphi$ on the boundary. A contradiction.

Thus, in the case of $\mathbb T^2$ and of $D^2$, it is meaningful to ask if there is a metric $g$ and a diffeomorphism $F$ (being the identity on the boundary in the case of $D^2$) such that $\mathcal C(g,F^*g)$ is empty. If one can find a vector field $f$ (vanishing on the boundary in the case of $D^2$) such that $\mathcal C(\jmath, f)=\varnothing$, then $\mathcal C(g,F^*_tg)=\varnothing$ for small $t\neq0$, as well, where $F_t$ is the time-$t$ map of the flow of $f$.

In the case of $\Sigma=\mathbb T^2$, we can readily find such a vector field for all conformal classes of complex structures. Indeed, let $\mathbb T^2=\mathbb C/\Gamma$ where $\Gamma$ is a lattice in $\mathbb C$ and let $\jmath$ be the complex structure on $\Sigma$ induced by that on $\mathbb C$. Up to Euclidean isometries and homotheties, we can assume that $\Gamma$ is generated by $1,\tau\in\mathbb C$, where $\tau=a+ib$ with $b>0$. Consider the vector field which in a global holomorphic trivialization of $T^{(1,0)}\Sigma$ is written as $f(z)=e^{\frac{2\pi i}{b}\mathrm{Im}\,z}$. Notice that $f$ is well-defined since it is invariant under translations by $1$ and $\tau$. Moreover,
\[
\bar{\partial}_\jmath f(z)=\frac{\partial}{\partial \bar z}e^{\frac{\pi}{b}(z-\bar z)}=-\frac{\pi}{b}f(z),
\] 
which is nowhere vanishing. 

However, we do not know if such a vector field $f$ exists on $D^2$. Since vector fields on $D^2$ correspond to functions in a global trivialization of $T^{(1,0)}D^2$, we have the following open question.
\begin{qst}\label{q:disc}
Does there exist a smooth function $f\colon D^2\to\mathbb C$ satisfying the following two conditions?
\begin{enumerate}  
\itemsep=1.5ex
\item $\forall z\in D^2,\quad \frac{\partial f}{\partial \bar z}(z)\neq0$.
\item $\forall z\in\partial D^2,\quad f(z)=0$.
\end{enumerate}
\end{qst}
\subsection{Plan of the paper}
Theorem \ref{t:main1} is proven in Section \ref{s:two}. Theorem \ref{t:main2} is proven in Section \ref{s:three}.
\subsection{Acknowledgments}
P.A.~and G.B.~are partially supported by the Deutsche Forschungsgemeinschaft 
under Germany's Excellence Strategy EXC2181/1 - 390900948 (the Heidelberg STRUCTURES Excellence Cluster), the Collaborative Research Center SFB/TRR 191 - 281071066 (Symplectic Structures in Geometry, Algebra and Dynamics), and the Research Training Group RTG 2229 - 281869850 (Asymptotic Invariants and Limits of Groups and Spaces). G.B.~warmly thanks Thomas Rot for stimulating discussions around the topics of this paper.
\section{Proof of Theorem \ref{t:main1}}\label{s:two}
We prove Theorem \ref{t:main1}.(1). Let $h$ be a symmetric bilinear two-tensor field over $\Sigma$ such that $\mathcal C(g,h)$ is finite and $\mathcal C(g,h)\subset \Sigma\setminus\partial \Sigma$. Recall the definition of $H$ and $H^a$ from \eqref{e:H} and \eqref{e:Ha}.

If $\Sigma$ has no boundary, then $[\mathcal C(g,h)]=e(E^a)=2\chi(\Sigma)$ by the Poincar\'e--Hopf Theorem for oriented plane bundles \cite[Theorem 11.17]{BT}. If $\Sigma$ has boundary, let $\hat\Sigma$ be the closed, oriented surface that we obtain from $\Sigma$ by gluing a disc $D_1,\ldots,D_n$ along each boundary component $C_1,\ldots,C_n$. The gluing maps $D^2\to D_i$ have the Euclidean disc
\[
D^2=\{(x,y)\in\R^2\ |\ x^2+y^2\leq 1\}
\]
as domain and send the boundary $\partial D^2$ traversed in the positive sense to $\bar C_i$, that is, to $C_i$ traversed in the negative sense. In this way, the gluing maps are positively oriented with respect to the orientation on $\hat \Sigma$.

We let $\hat g$ be any extension of $g$ to $\Sigma$ as a Riemannian metric. On the bundle $E^a|_{D_i}$ we choose a nowhere vanishing section $M^i$ defined as the reflection along the direction of $\partial_x\in TD^2$. Let $w_{\bar C_i}(H^a,M^i)$ be the winding number of $H^a$ with respect to $M^i$ along $C_i$ traversed in the negative direction. Then 
\[
w_{\bar C_i}(H^a,M^i)=w_{\bar C_i}(H^a,R^i)+w_{\bar C_i}(R^i,M^i)=-w_{C_i}(H^a,R^i)+w_{\partial D^2}(R^i,M^i)=-w_i(g,h)+2,
\]
where we have used that $\bar C_i$ is identified with $\partial D^2$ and that the unoriented line tangent to $\partial D^2$ rotates twice with respect to the horizontal unoriented line. By the Extension Theorem in \cite[p.~145]{GP}, it is possible to construct an extension $\hat h$ of $h$ to $\hat \Sigma$ such that $\mathcal C(\hat g,\hat h)=\mathcal C(g,h)\cup\{z_1,\ldots,z_n\}$, where $z_1,\ldots,z_n$ are the centers of the discs $D_1,\ldots,D_n$ and
\begin{equation}\label{e:z_i}
\mathrm{ind}_{(\hat g,\hat h)}(z_i)=w_{\bar C_i}(H^a,M^i)=2-w_i(g,h).
\end{equation}
Therefore,
\[
[\mathcal C(g,h)]=[\mathcal C(\hat g,\hat h)]-\sum_{i=1}^n\mathrm{ind}_{(\hat g ,\hat h)}(z_i)=2\chi(\hat\Sigma)-2n+\sum_{i=1}^nw_i(g,h)=2\chi(\Sigma)+\sum_{i=1}^nw_i(g,h),
\]
where we used that $\chi(\Sigma)+n=\chi(\hat\Sigma)$ as follows from the formula $\chi(A\cup B)=\chi(A)+\chi(B)-\chi(A\cap B)$. We have thus completed the proof of Theorem \ref{t:main1}.(1). 

Let us first prove Theorem \ref{t:main1}.(2) when $\Sigma$ has no boundary. Let us consider an embedded closed disc $D$ containing $\mathcal C$ in its interior. There is a section $H^{\mathrm{out}}$ of $E^a$ which is nowhere vanishing on $\Sigma\setminus \mathring{D}$ and there is a section $H^{\mathrm{in}}$ which is nowhere vanishing over $D$. The winding number of $H^{\mathrm{out}}$ with respect to $H^{\mathrm{in}}$ along $\partial D$ is $w(H^{\mathrm{out}},H^{\mathrm{in}})=2\chi(\Sigma)$. For each $z\in\mathcal C$ consider an embedded closed disc $D^z$ centered at $z$ and contained in $\mathring{D}$. After shrinking the discs $D^z$ we may assume that they are pairwise disjoint. Let $H^z$ be a section of $E^a|_{D^z}$ which has just one zero at $z$ with index $\mathrm{ind}(z)=\imath(z)$. Thus the winding number of $H^z$ with respect to $H^{\mathrm{in}}$ along $\partial D^z$ is $w(H^z,H^{\mathrm{in}})=\imath(z)$. Since $2\chi(\Sigma)=\sum_{z\in \mathcal C}\iota(z)$ by assumption, we get
\[
w(H^{\mathrm{out}},H^{\mathrm{in}})=\sum_{z\in \mathcal C}w(H^z,H^\mathrm{in}).
\]
Consider the surface
\[
\tilde \Sigma:=D\setminus\bigsqcup_{z\in \mathcal C}\mathring D^z.
\]
It satisfies $\partial\tilde\Sigma=\partial D\sqcup (\sqcup_{z\in\mathcal C}\overline{\partial D^z})$. Since $w(H^{\mathrm{out}},H^{\mathrm{in}})-\sum_{z\in \mathcal C}w(H^z,H^\mathrm{in})=0$, the Extension Theorem in \cite[p.~145]{GP} implies that there is a nowhere vanishing section $\tilde H$ of $E^a|_{\tilde \Sigma}$ coinciding with $H^{\mathrm{out}}$ on $\partial D$ and with $H^z$ on every $\partial D^z$. Thus, $H^{\mathrm{out}}$, $\tilde H$, and all $H^z$ glue together to yield a section $H$ of $E^a\to\Sigma$ having the desired properties.

When $\Sigma$ has boundary, we construct the closed surface $\hat\Sigma$ as in the proof of Theorem \ref{t:main1}.(1). We define $\hat{\mathcal C}:=\mathcal C\cup\{z_1,\ldots,z_n\}$ and $\hat\imath\colon\hat{\mathcal C}\to\Z$ as the extension of $\imath$ such that $\imath(z_i)=2-w_i$ for all $i=1,\ldots,n$. Applying Theorem \ref{t:main1}.(2) for closed surfaces to $\hat\Sigma$ and $\hat\imath$ and using \eqref{e:z_i} yields Theorem \ref{t:main1}.(2) for the case of surfaces with boundary, as well.

\section{Proof of Theorem \ref{t:main2}}\label{s:three}
Let $C_i$ be a component of $\partial \Sigma$ for some $i\in\{1,\ldots,n\}$. There is $j\in\{1,\ldots,n\}$ such that $F(C_i)=C_j$. Recall that $\di F|_{C_i}$ is expressed by the matrix 
\[
N_i=c_i\begin{pmatrix}
	a_i&0\\b_i&1
\end{pmatrix}
\]
with respect to the positive orthonormal bases $\nu_i,\tau_i$ and $\nu_j,\tau_j$.

The metric $F^*g|_{C_i}$ is represented by the endomorphism $\di F^T\cdot\di F$ via \eqref{e:H}. A computation shows that the matrix representing $\di F^T\cdot\di F$ with respect to the basis $\nu_i,\tau_i$ is
\[
N^T_iN_i=c^2_iQ_i, \quad\text{with}\quad Q_i=\begin{pmatrix}
	a^2_i+b^2_i&b_i\\
	b_i&1
\end{pmatrix}.
\]
We point out that the condition that $(a_i,b_i)$ is never equal to $(1,0)$ is equivalent to $Q_i$ having distinct eigenvalues, since $Q_i$ is symmetric. Let $q_i\colon C_i\to \R P^1\cong\R/\pi\Z$ be the eigendirection of $Q_i$ with larger eigenvalue. By \eqref{e:w}, $w_i(g,F^*g)$ is the degree of the map $q_i\colon C_i\to\R/\pi\Z$. Therefore, our goal is to show that the degree of $q_i$ is equal to the winding number of $(a_i-1,b_i)\colon C_i\to\R^2$ around the origin. To this purpose, let us parametrize $C_i$ in the positive direction by $\theta_i\in \R/2\pi\Z$ and, to ease notation, let us drop all the subscripts $i$ in what follows.

We may assume without loss of generality that the curve $(a-1,b)\colon\R/2\pi\Z\to\R^2$ intersects the positive real axis transversely. In this case $w(a-1,b)$ counts the number of points $\theta_0\in\R/2\pi\Z$ such that $(a(\theta_0)-1,b(\theta_0))$ lies on the positive real axis, namely $a(\theta_0)>1$ and $b(\theta_0)=0$, with sign: the intersection is counted positively if $b'(\theta_0)>0$ and negatively if $b'(\theta_0)<0$.

On the other hand, the degree of $q$ is computed using a regular value $\xi\in\R/\pi\Z$ of $q$. Being regular means that $q'(\theta_0)\neq0$ for all $\theta_0\in q^{-1}(\xi)$. In this case, the degree of $q$ counts number of points $\theta_0\in q^{-1}(\xi)$ with sign: the point $\theta_0$ is counted positively if $q'(\theta_0)>0$ and negatively if $q'(\theta_0)<0$. 

Choosing $\xi=0$, we see that $\theta_0\in q^{-1}(0)$ if and only if $(1,0)\in\R^2$ is an eigenvector of $Q$ with eigenvalue larger than $1$. This happens exactly when $b(\theta_0)=0$ and $a(\theta_0)>1$, that is when $(a-1,b)$ intersects the positive real axis. Therefore, we prove that $0$ is a regular value of $q$ and that $w(a-1,b)$ is the degree of $q$ if we can show that for every such $\theta_0$ the numbers $b'(\theta_0)$ and $q'(\theta_0)$ have the same sign. 

For this purpose, let $v(\theta)=(x(\theta),y(\theta))\in\R^2$ be a generator of the line $q(\theta)$ such that $v(\theta_0)=(1,0)$ and write $\lambda(\theta)$ for the corresponding eigenvalue of $Q(\theta)$, so that $\lambda(\theta_0)=a(\theta_0)$. Then $q'(\theta_0)=y'(\theta_0)$. To compute $y'(\theta_0)$ we differentiate the vector equation $\big(Q(\theta)-\lambda(\theta)I\big)v(\theta)=0$ at $\theta_0$:
\[
\big(Q(\theta_0)-\lambda(\theta_0)I\big)v'(\theta_0)+\big(Q'(\theta_0)-\lambda'(\theta_0)I\big)v(\theta_0)=0.
\]
Therefore, substituting the values for $Q(\theta_0)$, $\lambda(\theta_0)$ and $Q'(\theta_0)$ and taking the $y$-component of the vector equation, we get
\[
\big(1-a(\theta_0)\big)y'(\theta_0)+b'(\theta_0)=0.
\]
Thus,
\[
q'(\theta_0)=y'(\theta_0)=\frac{b'(\theta_0)}{a(\theta_0)-1}
\]
from which we see that $q'(\theta_0)$ and $b'(\theta_0)$ have the same sign since $a(\theta_0)>1$. This completes the proof.

\bibliography{biblio_conformal}
\bibliographystyle{plain}
\end{document}